\documentclass{article}
\usepackage{graphicx} 
\usepackage{subfigure}
\usepackage{amsmath, amssymb, amsthm}
\usepackage{mathrsfs}
\usepackage{algorithm, algorithmicx, algpseudocode}
\usepackage{cite}
\usepackage{url}
\usepackage{multirow, booktabs}

\newtheorem{problem}{Problem}
\newtheorem{theorem}{Theorem}
\newtheorem{proposition}{Proposition}
\newtheorem{definition}{Definition}
\newtheorem{assumption}{Assumption}

\title{Modified memoryless spectral-scaling Broyden family on Riemannian manifolds}

\author{Hiroyuki Sakai \and Hideaki Iiduka}

\begin{document}
\maketitle

\begin{abstract}
This paper presents modified memoryless quasi-Newton methods based on the spectral-scaling Broyden family on Riemannian manifolds.
The method involves adding one parameter to the search direction of the memoryless self-scaling Broyden family on the manifold.
Moreover, it uses a general map instead of vector transport.
This idea has already been proposed within a general framework of Riemannian conjugate gradient methods where one can use vector transport, scaled vector transport, or an inverse retraction.
We show that the search direction satisfies the sufficient descent condition under some assumptions on the parameters.
In addition, we show global convergence of the proposed method under the Wolfe conditions.
We numerically compare it with existing methods, including Riemannian conjugate gradient methods and the memoryless spectral-scaling Broyden family.
The numerical results indicate that the proposed method with the BFGS formula is suitable for solving an off-diagonal cost function minimization problem on an oblique manifold.
\end{abstract}

\section{Introduction}\label{sec:introduction}
Riemannian optimization has recently attracted a great deal of attention and has been used in many applications,
including low-rank tensor completion \cite{vandereycken2013low, kressner2014low},
machine learning \cite{nickel2017poincare},
and shape analysis \cite{huang2016riemannian}.

Iterative methods for solving unconstrained optimization problems on the Euclidean space have been studied for a long time \cite{nocedal1999numerical}.
Quasi-Newton methods and nonlinear conjugate gradient methods are the especially important ones and have been implemented in various software packages.

Here, quasi-Newton methods need to store dense matrices,
so it is difficult to apply them to large-scale problems.
Shanno \cite{shanno1978conjugate} proposed a memoryless quasi-Newton method as a way to deal with this problem.
This method \cite{kou2015modified, moyi2016sufficient, nakayama2018hybrid, nakayama2018memoryless, nakayama2019memoryless} has proven effective at solving large-scale unconstrained optimization problems.
The concept is simple:
an approximate matrix is updated by using the identity matrix instead of the previous approximate matrix.
Similar to the case of nonlinear conjugate gradient methods,
the search direction can be computed without having to use matrices and simply by taking the inner product without matrices.

Kou and Dai \cite{kou2015modified} proposed a modified memoryless spectral-scaling BFGS method.
Their method involves adding one parameter to the search direction of the memoryless self-scaling BFGS method.
In \cite{nakayama2018hybrid}, Nakayama used this technique to devise a memoryless spectral-scaling Broyden family.
In addition, he showed that the search direction is a sufficient descent direction and has the global convergence property.
Nakayama, Narushima, and Yabe \cite{nakayama2019memoryless} proposed memoryless quasi-Newton methods based on the spectral-scaling Broyden family \cite{chen2013spectral}.
Their methods generate a sufficient descent direction and have the global convergence property.

Many useful iterative methods for solving unconstrained optimization problems on manifolds have been studied (see \cite{absil2008optimization, sato2021riemannian}).
They have been obtained by extending iterative methods in Euclidean space by using the concepts of retraction and vector transport.
For example, Riemannian quasi-Newton methods \cite{huang2015broyden, huang2018riemannian} and Riemannian conjugate gradient methods \cite{sato2015new, zhu2020riemannian, sakai2020hybrid, sato2021riemannian} have been developed.
Sato and Iwai \cite{sato2015new} introduced scaled vector transport \cite[Definition 2.2]{sato2015new} in order to remove the assumption of isometric vector transport from the convergence analysis.
Zhu and Sato \cite{zhu2020riemannian} proposed Riemannian conjugate gradient methods that use an inverse retraction instead of vector transport.
In \cite{sato2021riemannian}, Sato proposed a general framework of Riemannian conjugate gradient methods.
This framework uses a general map instead of vector transport and utilizes the existing Riemannian conjugate gradient methods such as ones that use vector transport,
scaled vector transport \cite{sato2015new}, or inverse retraction \cite{zhu2020riemannian}.

In \cite{ring2012optimization}, Ring and Wirth proposed the BFGS method, which has a global convergence property under some convexity assumptions.
Narushima et al. \cite{narushima2023memoryless} proposed memoryless quasi-Newton methods based on the spectral-scaling Broyden family on Riemannian manifolds.
They extended the memoryless spectral-scaling Broyden family in Euclidean space to Riemannian manifolds with an additional modification to ensure a sufficient descent condition.
Moreover, they presented a global convergence analysis under the Wolfe conditions. In particular, they did not assume convexity of the objective function or isometric vector transport.
The results of the previous studies are summarized in Tables \ref{tab:quasi} and \ref{tab:memoryless}.

In this paper, we propose a modified memoryless quasi-Newton method based on the spectral-scaling Broyden family on Riemannian manifolds, exploiting the idea used in the paper \cite{nakayama2018hybrid}.

In the case of Euclidean space, Nakayama \cite{nakayama2018hybrid} reported that
the modified memoryless quasi-Newton method based on the spectral-scaling Broyden family shows good experimental performance with parameter tuning.
Therefore, it is worth extending it to Riemannian manifolds.
Our method is based on the memoryless quasi-Newton methods on Riemannian manifolds proposed by Narushima et al. \cite{narushima2023memoryless} as well as on the modification by Kou and Dai \cite{kou2015modified}.
It uses a general map to transport vectors similarly to the general framework of Riemannian conjugate gradient methods \cite{sato2022riemannian}.
This generalisation allows us to use maps such as an inverse retraction \cite{zhu2020riemannian} instead of vector transport.
We show that our method generates a search direction satisfying the sufficient descent condition under some assumptions on the parameters (see Proposition \ref{prop:sufficient}).
Moreover, we present global convergence analyses under the Wolfe conditions (see Theorem \ref{thm:convergence}).
Furthermore, we describe the results of numerical experiments comparing our method with the existing ones,
including Riemannian conjugate gradient methods \cite{sakai2020hybrid} and the memoryless spectral-scaling Broyden family on Riemannian manifolds \cite{narushima2023memoryless}.
The key advantages of the proposed methods are
the added parameter $\xi_{k-1}$ and support for maps other than vector transports.
As shown in the numerical experiments,
the proposed method may outperform the existing methods depending on how the parameter $\xi_{k-1}$ is chosen.
It has an advantage over \cite{narushima2023memoryless} in that it can use a map such as an inverse retraction, which is not applicable in \cite{narushima2023memoryless}.

This paper is organized as follows.
Section \ref{sec:preliminaries} reviews the fundamentals of Riemannian geometry and Riemannian optimization.
Section \ref{sec:method} proposes the modified memoryless quasi-Newton method based on the spectral-scaling Broyden family.
Section \ref{sec:analysis} gives a global convergence analysis.
Section \ref{sec:numerical} compares the proposed method with the existing methods through numerical experiments.
Section \ref{sec:conclusion} concludes the paper.

\begin{table}[htbp]
\caption{Results of previous studies on Quasi-Newton methods in Euclidean space and Riemannian manifolds.\label{tab:quasi}}
\centering
\begin{tabular}{lccc}
\bottomrule
& \multirow{2}{*}{BFGS} & \multirow{2}{*}{Broyden family} & spectral-scaling \\
& & & Broyden family \\
\hline
\multirow{2}{*}{Euclidean} & \multirow{2}{*}{------} & \multirow{2}{*}{------} & Chen--Cheng \\
& & & (2013) \cite{chen2013spectral} \\
\hline
\multirow{4}{*}{Riemannian} & Ring--Wirth & & \multirow{4}{*}{------} \\
& (2012) \cite{ring2012optimization} & Huang et al. & \\
& Huang et al. & (2015) \cite{huang2015broyden} & \\
& (2018) \cite{huang2018riemannian} & & \\
\toprule
\end{tabular}
\end{table}

\begin{table}[htbp]
\caption{Results of previous studies and ours on memoryless quasi-Newton methods in Euclidean space and Riemannian manifolds.\label{tab:memoryless}}
\centering
\begin{tabular}{lcc}
\bottomrule
 & spectral-scaling & modified spectral-scaling \\
 & Broyden family & Broyden family \\
\hline
\multirow{2}{*}{Euclidean}  & Nakayama et al. & Nakayama \\
 & (2019) \cite{nakayama2019memoryless} & (2018) \cite{nakayama2018hybrid} \\
\hline
\multirow{2}{*}{Riemannian} & Narushima et al. & \multirow{2}{*}{this work} \\
 & (2023) \cite{narushima2023memoryless} & \\
\toprule
\end{tabular}
\end{table}

\section{Mathematical preliminaries}\label{sec:preliminaries}
Let $M$ be a Riemannian manifold with Riemannian metric $g$.
$T_xM$ denotes the tangent space of $M$ at a point $x\in M$.
The tangent bundle of $M$ is denoted by $TM$.
A Riemannian metric at $x\in M$ is denoted by
$\langle\cdot,\cdot\rangle_x:T_xM\times T_xM\to\mathbb{R}$.
The induced norm of a tangent vector $\eta\in T_xM$ is defined by $\lVert\eta\rVert_x:=\sqrt{\langle\eta,\eta\rangle_x}$.
For a given tangent vector $\eta\in T_xM$,
$\eta^\flat$ represents the flat of $\eta$, i.e.,
$\eta^\flat:T_xM\to\mathbb{R}:\xi\mapsto\langle\eta,\xi\rangle_x$.
Let $F:M\to N$ be a smooth map between smooth manifolds;
then, the derivative of $F$ at $x\in M$ is denoted by $\mathrm{D}F(x):T_xM\to T_{F(x)}N$.
For a smooth function $f:M\to\mathbb{R}$,
$\mathrm{grad}f(x)$ denotes the Riemannian gradient at $x\in M$,
i.e., a unique element of $T_xM$ satisfying
\begin{align*}
    \langle\mathrm{grad}f(x),\eta\rangle_x=\mathrm{D}f(x)[\eta],
\end{align*}
for all $\eta\in T_xM$. $\mathrm{Hess}f(x)$ denotes the Riemannian Hessian at $x\in M$, which is defined as
\begin{align*}
    \mathrm{Hess}f(x):T_xM\to T_xM:\eta\mapsto\nabla_\eta\mathrm{grad}f(x),
\end{align*}
where $\nabla$ denotes the Levi-Civita connection of $M$
(see \cite{absil2008optimization}).

\begin{definition}
Any smooth map $R:TM\to M$ is called a retraction on $M$
if it has the following properties.
\begin{itemize}
    \item $R_x(0_x)=x$, where $0_x$ denotes the zero element of $T_xM$;
    \item $\mathrm{D}R_x(0_x)=\mathrm{id}_{T_xM}$ with the canonical identification $T_{0_x}(T_xM)\simeq T_xM$,
\end{itemize}
where $R_x$ denotes the restriction of $R$ to $T_xM$.
\end{definition}

\begin{definition}
Any smooth map $\mathcal{T}:TM\oplus TM\to TM$ is called a vector transport on $M$ if it has the following properties.
\begin{itemize}
\item There exists a retraction $R$ such that $\mathcal{T}_\eta(\xi)\in T_{R_x(\eta)}M$ for all $x\in M$ and $\eta,\xi\in T_xM$;
\item $\mathcal{T}_{0_x}(\xi)=\xi$ for all $x\in M$ and $\xi\in T_xM$;
\item $\mathcal{T}_\eta(a\xi+b\zeta)=a\mathcal{T}_\eta(\xi)+b\mathcal{T}_\eta(\zeta)$ for all $x\in M$,
$a,b\in\mathbb{R}$ and $\eta,\xi,\zeta\in T_xM$,
\end{itemize}
where $\mathcal{T}_\eta(\xi):=\mathcal{T}(\eta,\xi)$.
\end{definition}

Let us consider an iterative method in Riemannian optimization.
For an initial point $x_0\in M$, step size $\alpha_k>0$,
and search direction $\eta_k\in T_{x_k}M$,
the $k$-th approximation to the solution is described as
\begin{align}\label{eq:iterative}
  x_{k+1}=R_{x_k}(\alpha_k\eta_k),
\end{align}
where $R$ is a retraction.
We define $g_k:=\mathrm{grad}f(x_k)$.
Various algorithms have been developed to determine the search direction $\eta_k$.
We say that $\eta_k$ is a sufficient descent direction
if the sufficient descent condition,
\begin{align}\label{eq:s-descent}
    \langle g_k,\eta_k\rangle_{x_k}\leq-\kappa\lVert g_k\rVert_{x_k}^2
\end{align}
holds for some constant $\kappa>0$.

In \cite{huang2015broyden, huang2018riemannian, narushima2023memoryless},
the search direction $\eta_k\in T_{x_k}M$ of Riemannian quasi-Newton methods
is computed as
\begin{align}\label{eq:search}
    \eta_k=-\mathcal{H}_k[g_k],
\end{align}
where $\mathcal{H}_k:T_{x_k}M\to T_{x_k}M$
is a symmetric approximation to $\mathrm{Hess}f(x_k)^{-1}$.

In \cite{sato2022riemannian},
Sato proposed a general framework of Riemannian conjugate gradient methods
by using a map $\mathscr{T}^{(k-1)}:T_{x_{k-1}}M\to T_{x_k}M$
which satisfies Assumption \ref{asm:map},
to transport $\eta_{k-1}\in T_{x_{k-1}}M$ to $T_{x_k}M$;
i.e., the search direction $\eta_k$ is computed as
\begin{align*}
    \eta_k=-g_k+\beta_k\sigma_{k-1}\mathscr{T}^{(k-1)}(\eta_{k-1}),
\end{align*}
where $\beta_k\in\mathbb{R}$, and $\sigma_{k-1}$ is a scaling parameter (see \cite[Section 4.1]{sato2022riemannian}) satisfying
\begin{align*}
    0<\sigma_{k-1}\leq\min\left\{1,\frac{\lVert\eta_{k-1}\rVert_{x_{k-1}}}{\lVert\mathscr{T}^{(k-1)}(\eta_{k-1})\rVert_{x_k}}\right\}.
\end{align*}

\begin{assumption}\label{asm:map}
    There exist $C \geq  0$ and $K\subset\mathbb{N}$, such that for all $k\in K$,
    \begin{align}\label{eq:mapStronger}
        \lVert\mathscr{T}^{(k)}(\eta_k)-\mathrm{D}R_x(\alpha_k\eta_k)[\eta_k]\rVert_{x_{k+1}}\leq C\alpha_k\lVert\eta_k\rVert_{x_k}^2,
    \end{align}
    and for all $k\in\mathbb{N}-K$,
    \begin{align}\label{eq:mapWeaker}
        \lVert\mathscr{T}^{(k)}(\eta_k)-\mathrm{D}R_x(\alpha_k\eta_k)[\eta_k]\rVert_{x_{k+1}}\leq C(\alpha_k+\alpha_k^2)\lVert\eta_k\rVert_{x_k}^2.
    \end{align}
\end{assumption}

Note that inequality \eqref{eq:mapWeaker} is weaker than \eqref{eq:mapStronger}.
For $k$ satisfying the stronger condition \eqref{eq:mapStronger},
the assumption of Theorem \ref{thm:zoutendijk} can be weakened.
Further details can be found in \cite[Remark 4.3]{sato2022riemannian}.
Assumption \ref{asm:map} requires that $\mathscr{T}^{(k)}$ is an approximation of the differentiated retraction.
Therefore, the differentiated retraction clearly satisfies the conditions of Assumption \ref{asm:map}.
In \cite[Example 4.5]{sato2022riemannian} and \cite[Example 4.6]{sato2022riemannian},
Sato gives examples of maps $\mathscr{T}^{(k)}$ satisfying Assumption \ref{asm:map} in the case of the unit sphere and Grassmann manifolds, respectively.
In \cite[Proposition 1]{zhu2020riemannian}, Zhu and Sato proved that the inverse of the retraction satisfies Assumption \ref{asm:map}.

Sato \cite[Section 4.3]{sato2022riemannian} generalized the parameter $\beta_k$ (i.e., Fletcher--Reeves (FR) \cite{sato2015new}, Dai--Yuan (DY) \cite{sato2016dai}, Polak--Ribi\`ere--Polyak (PRP) and Hestenes--Stiefel (HS) methods) as follows:
\begin{align}
    \beta_k^\mathrm{FR}&=\frac{\lVert g_k\Vert_{x_k}^2}{\lVert g_{k-1}\rVert_{x_{k-1}}^2},\nonumber\\
    \beta_k^\mathrm{DY}&=\frac{\lVert g_k\Vert_{x_k}^2}{\langle g_k,\sigma_{k-1}\mathscr{T}^{(k-1)}(\eta_{k-1})\rangle_{x_k}-\langle g_{k-1},\eta_{k-1}\rangle_{x_{k-1}}},\label{eq:DY}\\
    \beta_k^\mathrm{PRP}&=\frac{\lVert g_k\rVert_{x_k}^2-\langle g_k,l_{k-1}\mathscr{S}^{(k-1)}(g_{k-1})\rangle_{x_{k-1}}}{\lVert g_{k-1}\rVert_{x_{k-1}}^2},\nonumber\\
    \beta_k^\mathrm{HS}&=\frac{\lVert g_k\rVert_{x_k}^2-\langle g_k,l_{k-1}\mathscr{S}^{(k-1)}(g_{k-1})\rangle_{x_{k-1}}}{\langle g_k,\sigma_{k-1}\mathscr{T}^{(k-1)}(\eta_{k-1})\rangle_{x_k}-\langle g_{k-1},\eta_{k-1}\rangle_{x_{k-1}}},\nonumber
\end{align}
where $l_{k-1}>0$ and $\mathscr{S}^{(k-1)}:T_{x_{k-1}}M\to T_{x_k}M$ is an appropriate mapping. Therefore, we can use the Hager-Zhang (HZ) methods \cite[Section 3]{sakai2023global} generalized by the above techniques, as follows:
\begin{align}\label{eq:HZ}
    \beta_k^\mathrm{HZ}=\beta_k^\mathrm{HS}-\mu\frac{\lVert y_{k-1}\rVert_{x_k}^2\langle g_k,\mathscr{T}^{(k-1)}(\eta_{k-1})\rangle_{x_k}}{\left(\langle g_k,\sigma_{k-1}\mathscr{T}^{(k-1)}(\eta_{k-1})\rangle_{x_k}-\langle g_{k-1},\eta_{k-1}\rangle_{x_{k-1}}\right)^2},
\end{align}
where $\mu>1/4$.

We suppose that the search direction $\eta_k\in T_{x_k}M$ is a descent direction.
In \cite[Section 4.4]{sato2022riemannian}, Sato introduced the Riemannian version of the Wolfe conditions with a $\mathscr{T}^{(k)}:T_{x_k}M\to T_{x_{k+1}}M$,
called $\mathscr{T}^{(k)}$-Wolfe conditions.
$\mathscr{T}^{(k)}$-Wolfe conditions are written as
\begin{align}
    &f(R_{x_k}(\alpha_k\eta_k))\leq f(x_k)+c_1{\alpha_k}\langle g_k,\eta_k\rangle_{x_k},\label{eq:armijo}\\
    &\langle\mathrm{grad}f(R_{x_k}(\alpha_k\eta_k)),\mathscr{T}^{(k)}(\eta_k)\rangle_{R_{x_k}(\alpha_k\eta_k)}\geq c_2\langle g_k,\eta_k\rangle_{x_k}, \label{eq:wolfe}
\end{align}
where $0<c_1<c_2<1$.
Note that the existence of a step size $\alpha_k>0$ satisfying the $\mathscr{T}^{(k)}$-Wolfe conditions is discussed in
\cite[Section 4.4]{sato2022riemannian}.
Moreover, algorithms \cite[Algorithm 3]{sakai2021sufficient} and \cite[Section 5.1]{sato2016dai} exist for finding step sizes which satisfy the $\mathscr{T}^{(k)}$-Wolfe conditions.

\section{Memoryless spectral-scaling Broyden family}\label{sec:method}
Let us start by reviewing the memoryless spectral-scaling Broyden family in Euclidean space.
In the Euclidean setting, an iterative optimization algorithm updates the current iterate $x_k$ to the next iterate $x_{k+1}$ with the updating formula,
\begin{align*}
    x_{k+1}=x_k+\alpha_kd_k,
\end{align*}
where $\alpha_k>0$ is a positive step size.
One often chooses a step size $\alpha_k>0$ to satisfy the Wolfe conditions (see \cite{wolfe1969, wolfe1971}),
\begin{align*}
    f(x_k+\alpha_kd_k)\leq f(x_k)+c_1\alpha_k\nabla f(x_k)^\top d_k, \\
    \nabla f(x_k+\alpha_kd_k)^\top d_k\geq c_2\nabla f(x_k)\top d_k,
\end{align*}
where $0<c_1<c_2<1$.
The search direction $d_k$ of the quasi-Newton methods is defined by
\begin{align}\label{eq:Eu-search}
    d_k=-H_k\nabla f(x_k),
\end{align}
where $g_k=\nabla f(x_k)$ and $H_k$ is a symmetric approximation to $\nabla^2f(x_k)^{-1}$. In this paper,
we will focus on the Broyden family, written as
\begin{align}~\label{eq:Eu-Broyden}
    \begin{split}
    H_k&=H_{k-1}-\frac{H_{k-1}y_{k-1}y_{k-1}^\top H_{k-1}}{y_{k-1}^\top H_{k-1}y_{k-1}}+\frac{s_{k-1}s_{k-1}^\top}{s_{k-1}^\top y_{k-1}}\\
    &\quad +\phi_{k-1}y_{k-1}^\top H_{k-1}y_{k-1}w_{k-1}w_{k-1}^\top,
    \end{split}
\end{align}
where
\begin{align*}
    w_{k-1}=\frac{s_{k-1}}{s_{k-1}^\top y_{k-1}}-\frac{H_{k-1}y_{k-1}}{y_{k-1}^\top H_{k-1}y_{k-1}},
\end{align*}
$s_{k-1}=x_k-x_{k-1}$ and $y_{k-1}=\nabla f(x_k)-\nabla f(x_{k-1})$.
$\phi_{k-1}$ is a parameter,
which becomes the DFP formula when $\phi_{k-1}=0$ or the BFGS formula
when $\phi_{k-1}=1$ (see \cite{nocedal1999numerical, sun2006optimization}).
Here, if $\phi_{k-1}\in[0,1]$,
then \eqref{eq:Eu-Broyden} is a convex combination of the DFP formula and the BFGS formula;
we call this interval the convex class.
Zhang and Tewarson \cite{zhang1988quasi} found a better choice in the case $\phi_{k-1}>1$;
we call this interval the preconvex class.
In \cite{chen2013spectral}, Chen and Cheng proposed the Broyden family based on the spectral-scaling secant condition \cite{cheng2010spectral} as follows:
\begin{align}\label{eq:Eu-ssBroyden}
    \begin{split}
        H_k&=H_{k-1}-\frac{H_{k-1}y_{k-1}y_{k-1}^\top H_{k-1}}{y_{k-1}^\top H_{k-1}y_{k-1}}+\frac{1}{\tau_{k-1}}\frac{s_{k-1}s_{k-1}^\top}{s_{k-1}^\top y_{k-1}} \\
        &\quad +\phi_{k-1}y_{k-1}^\top H_{k-1}y_{k-1}w_{k-1}w_{k-1}^\top,
    \end{split}
\end{align}
where $\tau_{k-1}>0$ is a spectral-scaling parameter.

Shanno \cite{shanno1978conjugate} proposed memoryless quasi-Newton methods
in which $H_{k-1}$ is replaced with the identity matrix in \eqref{eq:Eu-Broyden}.
Memoryless quasi-Newton methods avoid having to make memory storage for matrices and can solve large-scale unconstrained optimization problems.
In addition, Nakayama, Narushima and Yabe \cite{nakayama2019memoryless} proposed memoryless quasi-Newton methods based on
the spectral-scaling Broyden family by replacing $H_{k-1}$ with the identity matrix in \eqref{eq:Eu-ssBroyden},
i.e.,
\begin{align}\label{eq:Eu-mlBroyden}
    H_k=I-\frac{y_{k-1}y_{k-1}^\top}{y_{k-1}^\top y_{k-1}}+\frac{1}{\tau_{k-1}}\frac{s_{k-1}s_{k-1}^\top}{s_{k-1}^\top y_{k-1}}+\phi_{k-1}y_{k-1}^\top y_{k-1}w_{k-1}w_{k-1}^\top,
\end{align}
where
\begin{align*}
    w_{k-1}=\frac{s_{k-1}}{s_{k-1}^\top y_{k-1}}-\frac{y_{k-1}}{y_{k-1}^\top y_{k-1}}.
\end{align*}
From \eqref{eq:Eu-search} and \eqref{eq:Eu-ssBroyden},
the search direction $d_k$ of memoryless quasi-Newton methods based on
the spectral-scaling Broyden family can be computed as
\begin{align*}
    d_k&=-g_k+\left(\phi_{k-1}\frac{y_{k-1}^\top g_k}{d_{k-1}^\top y_{k-1}}-\left(\frac{1}{\tau_{k-1}}+\phi_{k-1}\frac{y_{k-1}^\top y_{k-1}}{s_{k-1}^\top y_{k-1}}\right)\frac{s_{k-1}^\top g_k}{d_{k-1}^\top y_{k-1}}\right)d_{k-1}\\
    &+\left(\phi_{k-1}\frac{d_{k-1}^\top g_k}{d_{k-1}^\top y_{k-1}}+\left(1-\phi_{k-1}\right)\frac{y_{k-1}^\top g_k}{y_{k-1}^\top y_{k-1}}\right)y_{k-1}.
\end{align*}
In \cite{nakayama2019memoryless}, they also proved global convergence for step sizes satisfying
the Wolfe conditions (see \cite[Theorem 3.1]{nakayama2019memoryless} and \cite[Theorem 3.6]{nakayama2019memoryless}).
In \cite{kou2015modified}, Kou and Dai proposed a modified memoryless self-scaling BFGS method and showed that
it generates a search direction satisfying the sufficient descent condition. Moreover, Nakayama \cite{nakayama2018hybrid} used the modification
by Kou and Dai and proposed a search direction $d_k$ defined by
\begin{align*}
    d_k&=-g_k+\left(\phi_{k-1}\frac{y_{k-1}^\top g_k}{d_{k-1}^\top y_{k-1}}-\left(\frac{1}{\tau_{k-1}}+\phi_{k-1}\frac{y_{k-1}^\top y_{k-1}}{s_{k-1}^\top y_{k-1}}\right)\frac{s_{k-1}^\top g_k}{d_{k-1}^\top y_{k-1}}\right)d_{k-1}\\
    &+\xi_{k-1}\left(\phi_{k-1}\frac{d_{k-1}^\top g_k}{d_{k-1}^\top y_{k-1}}+\left(1-\phi_{k-1}\right)\frac{y_{k-1}^\top g_k}{y_{k-1}^\top y_{k-1}}\right)y_{k-1},
\end{align*}
where $\xi_{k-1}\in[0,1]$ is a parameter.

\subsection{Memoryless spectral-scaling Broyden family on Riemannian manifolds}
We define $s_{k-1}=\mathcal{T}_{\alpha_{k-1}\eta_{k-1}}(\alpha_{k-1}\eta_{k-1})$ and $y_{k-1}=g_k-\mathcal{T}_{\alpha_{k-1}\eta_{k-1}}(g_{k-1})$.
The Riemannian quasi-Newton method with the spectral-scaling Broyden family \cite[(23)]{narushima2023memoryless} is written as
\begin{align}\label{eq:ssBroyden}
\begin{split}
    \mathcal{H}_k&=\tilde{\mathcal{H}}_{k-1}-\frac{\tilde{\mathcal{H}}_{k-1}y_{k-1}(\tilde{\mathcal{H}}_{k-1}y_{k-1})^\flat}{(\tilde{\mathcal{H}}_{k-1}y_{k-1})^\flat y_{k-1}}+\frac{1}{\tau_{k-1}}\frac{s_{k-1}s_{k-1}^\flat}{s_{k-1}^\flat y_{k-1}}\\
    &\quad+\phi_{k-1}(\tilde{\mathcal{H}}_{k-1}y_{k-1})^\flat y_{k-1}w_{k-1}w_{k-1}^\flat,
\end{split}
\end{align}
where
\begin{align*}
    w_{k-1}=\frac{s_{k-1}}{s_{k-1}^\flat y_{k-1}}-\frac{\tilde{\mathcal{H}}_{k-1}y_{k-1}}{(\tilde{\mathcal{H}}_{k-1}y_{k-1})^\flat y_{k-1}},
\end{align*}
and
\begin{align*}
    \tilde{\mathcal{H}}_{k-1}=\mathcal{T}_{\alpha_{k-1}\eta_{k-1}}\circ\mathcal{H}_{k-1}\circ(\mathcal{T}_{\alpha_{k-1}\eta_{k-1}})^{-1}.
\end{align*} 
Here, $\phi_{k-1}\geq 0$ is a parameter, and $\tau_{k-1}>0$ is a spectral-scaling parameter.

The idea of behind the memoryless spectral-scaling Broyden family is very simple:
replace $\tilde{\mathcal{H}}_{k-1}$ with $\mathrm{id}_{T_{x_{k-1}}M}$.
In \cite{narushima2023memoryless}, a memoryless spectral-scaling Broyden family on a Riemannian manifold is proposed by replacing $\Tilde{\mathcal{H}}_{k-1}$ with $\mathrm{id}_{T_{x_k}M}$.
To guarantee global convergence, they replaced $y_{k-1}$ by $z_{k-1}\in T_{x_k}M$
satisfying the following conditions \cite[(27)]{narushima2023memoryless}:
for positive constants $\underline{\nu},\overline{\nu}>0$,
\begin{align}
    \underline{\nu}\lVert s_{k-1}\rVert_{x_k}^2&\leq s_{k-1}^\flat z_{k-1},\label{eq:z-1}\\
    \lVert z_{k-1}\rVert_{x_k}&\leq \overline{\nu}\lVert s_{k-1}\rVert_{x_k}.\label{eq:z-2}
\end{align}
Here, we can choose $z_{k-1}$ by using Li-Fukushima regularization \cite{li2001modified}, which is a Levenberg–Marquardt type of regularization, and set
\begin{align}\label{eq:li-fukushima-1}
    z_{k-1}=y_{k-1}+\nu_{k-1}s_{k-1},
\end{align}
where
\begin{align}\label{eq:li-fukushima-2}
    \nu_{k-1}=\begin{cases}0,&\text{if }s_{k-1}^\flat y_{k-1}\geq\hat{\nu}\lVert s_{k-1}\rVert_{x_k}^2,\\
    \max\left\{0,-\dfrac{s_{k-1}^\flat y_{k-1}}{\lVert s_{k-1}\rVert_{x_k}^2}\right\}+\hat{\nu},&\text{otherwise},\end{cases}
\end{align}
and $\hat{\nu}>0$.
We can also use Powell’s damping technique \cite{nocedal1999numerical}, which sets
\begin{align}\label{eq:powell-1}
    z_{k-1}=\nu_{k-1}y_{k-1}+(1-\nu_{k-1})s_{k-1},
\end{align}
where $\hat{\nu}\in(0,1)$ and
\begin{align}\label{eq:powell-2}
    \nu_{k-1}=\begin{cases}1,&\text{if }s_{k-1}^\flat y_{k-1}\geq\hat{\nu}\lVert s_{k-1}\rVert_{x_k}^2,\\
    \dfrac{(1-\hat{\nu})\lVert s_{k-1}\rVert_{x_k}^2}{\lVert s_{k-1}\rVert_{x_k}^2-s_{k-1}^\flat y_{k-1}},&\text{otherwise}.\end{cases}
\end{align}
The proof that these choices satisfy conditions \eqref{eq:z-1} and \eqref{eq:z-2} is given in \cite[Proposition 4.1]{narushima2023memoryless}.
Thus, a memoryless spectral-scaling Broyden family on a Riemannian manifold \cite[(28)]{narushima2023memoryless} can be described as
\begin{align*}
    \mathcal{H}_k&=\gamma_{k-1}\mathrm{id}_{T_{x_kM}}-\gamma_{k-1}\frac{z_{k-1}z_{k-1}^\flat}{z_{k-1}^\flat z_{k-1}}+\frac{1}{\tau_{k-1}}\frac{s_{k-1}s_{k-1}^\flat}{s_{k-1}^\flat z_{k-1}} \\
    &\quad+\phi_{k-1}\gamma_{k-1}z_{k-1}^\flat z_{k-1}w_{k-1}w_{k-1}^\flat,
\end{align*}
where
\begin{align*}
    w_{k-1}=\frac{s_{k-1}}{s_{k-1}^\flat z_{k-1}}-\frac{z_{k-1}}{z_{k-1}^\flat z_{k-1}}.
\end{align*}
Here, $\gamma_{k-1}>0$ is a sizing parameter.
From \eqref{eq:search}, we can compute the search direction of the memoryless spectral-scaling Broyden family on a Riemannian manifold as follows:
\begin{align*}
    \eta_k=&-\gamma_{k-1}g_k \\
    &+\gamma_{k-1}\left(\phi_{k-1}\frac{z_{k-1}^\flat g_k}{s_{k-1}^\flat z_{k-1}}-\left(\frac{1}{\gamma_{k-1}\tau_{k-1}}+\phi_{k-1}\frac{z_{k-1}^\flat z_{k-1}}{s_{k-1}^\flat z_{k-1}}\right)\frac{s_{k-1}^\flat g_k}{s_{k-1}^\flat z_{k-1}}\right)s_{k-1} \\
    &+\gamma_{k-1}\left(\phi_{k-1}\frac{s_{k-1}^\flat g_k}{s_{k-1}^\flat z_{k-1}}+\left(1-\phi_{k-1}\right)\frac{z_{k-1}^\flat g_k}{z_{k-1}^\flat z_{k-1}}\right)z_{k-1}.
\end{align*}

\subsection{Proposed algorithm}
Let $\mathscr{T}^{(k-1)}:T_{x_{k-1}}M\to T_{x_k}M$ be a map which satisfies Assumption \ref{asm:map}.
Furthermore, we define $y_{k-1}=g_k-\mathscr{T}^{(k-1)}(g_{k-1})$ and $s_{k-1}=\mathscr{T}^{(k-1)}(\alpha_{k-1}\eta_{k-1})$.
We propose the following search direction of the modified memoryless spectral-scaling Broyden family on a Riemannian manifold:
\begin{align}\label{eq:p-search}
\begin{split}
    \eta_k=&-\gamma_{k-1}g_k \\
    &+\gamma_{k-1}\left(\phi_{k-1}\frac{z_{k-1}^\flat g_k}{s_{k-1}^\flat z_{k-1}}-\left(\frac{1}{\gamma_{k-1}\tau_{k-1}}+\phi_{k-1}\frac{z_{k-1}^\flat z_{k-1}}{s_{k-1}^\flat z_{k-1}}\right)\frac{s_{k-1}^\flat g_k}{s_{k-1}^\flat z_{k-1}}\right)s_{k-1} \\
    &+\gamma_{k-1}\xi_{k-1}\left(\phi_{k-1}\frac{s_{k-1}^\flat g_k}{s_{k-1}^\flat z_{k-1}}+\left(1-\phi_{k-1}\right)\frac{z_{k-1}^\flat g_k}{z_{k-1}^\flat z_{k-1}}\right)z_{k-1},
\end{split}
\end{align}
where $\xi_{k-1}\in[0,1]$ is a parameter, and $z_{k-1}\in T_{x_k}M$ is a tangent vector satisfying \eqref{eq:z-1} and \eqref{eq:z-2}.
Note that equation \eqref{eq:p-search} has not only added $\xi_{k-1}$,
but also changed the definition of the two tangent vectors $y_{k-1}$ and $s_{k-1}$ for determining $z_{k-1}$.
The proposed algorithm is listed in Algorithm \ref{alg:memoryless}.
Note that Algorithm \ref{alg:memoryless} is a generalization of memoryless quasi-Newton methods based on the spectral-scaling Broyden family proposed in \cite{narushima2023memoryless}.
In fact, if $\xi_{k-1}=1$ and $\mathscr{T}^{(k-1)}=\mathcal{T}_{\alpha_{k-1}\eta_{k-1}}(\cdot)$, then Algorithm \ref{alg:memoryless} coincides with it.

\begin{algorithm}
    \caption{Modified memoryless quasi‐Newton methods based on spectral-scaling Broyden family on Riemannian manifolds. \label{alg:memoryless}}
    \begin{algorithmic}[1]
        \Require Initial point $x_0 \in M$, $(\gamma_k)_{k=0}^\infty\subset (0,\infty)$, $(\phi_k)_{k=0}^\infty\subset [0,\infty)$, $(\xi_k)_{k=0}^\infty\subset [0,1]$, $(\tau_k)_{k=0}^\infty\subset (0,\infty)$.
        \Ensure Sequence $(x_k)_{k=0}^{\infty} \subset M$.
        \State $k \gets 0$.
        \State Set $\eta_0=-g_0=-\mathrm{grad}f(x_0)$.
        \Loop
        \State Compute a step size $\alpha_k>0$ satisfying the Wolfe conditions \eqref{eq:armijo} and \eqref{eq:wolfe}.
        \State Set $x_{k+1}=R_{x_k}(\alpha_k\eta_k)$.
        \State Compute $g_{k+1}:=\mathrm{grad}f(x_{k+1})$.
        \State Compute a search direction $\eta_{k+1}\in T_{x_{k+1}}M$ by \eqref{eq:p-search}.
        \State $k \gets k+1$
        \EndLoop
    \end{algorithmic}
\end{algorithm}

\section{Convergence analysis}\label{sec:analysis}
\begin{assumption}\label{asm:zoutendijk}
Let $f:M\to\mathbb{R}$ be a smooth, bounded below function with the following property:
there exists $L>0$ such that
    \begin{align*}
        \lvert\mathrm{D}(f\circ R_x)(t\eta)[\eta]-\mathrm{D}(f\circ R_x)(0_x)[\eta]\rvert\leq Lt,
    \end{align*}
for all $\eta\leq T_xM$, $\lVert\eta\rVert_x=1$, $x\in M$ and $t\geq 0$.
\end{assumption}

\begin{assumption}\label{asm:bound}
We suppose that there exists $\Gamma>0$ such that
    \begin{align*}
        \lVert g_k\rVert_{x_k}\leq\Gamma
    \end{align*}
for all $k$.
\end{assumption}

Zoutendijk’s theorem
about the $\mathscr{T}^{(k)}$-Wolfe conditions \cite[Theorem 5.3]{sato2022riemannian},
is described as follows:

\begin{theorem}\label{thm:zoutendijk}
Suppose that Assumptions \ref{asm:map} and \ref{asm:zoutendijk} hold.
Let $(x_k)_{k=0,1,\cdots}$ be a sequence generated by an iterative method of the form \eqref{eq:iterative}.
We assume that the step size $\alpha_k$ satisfies the $\mathscr{T}^{(k)}$-Wolfe conditions \eqref{eq:armijo} and \eqref{eq:wolfe}.
If the search direction $\eta_k$ is a descent direction and
there exists $\mu>0$,
such that $\eta_k$ satisfies $\lVert g_k\rVert_{x_k}\leq\mu\lVert\eta_k\rVert_{x_k}$ for all $k\in\mathbb{N}-K$,
then the following holds:
    \begin{align*}
        \sum_{k=0}^\infty\frac{\langle g_k,\eta_k\rangle_{x_k}^2}{\lVert\eta_k\rVert_{x_k}^2}<\infty,
    \end{align*}
where $K$ is the subset of $\mathbb{N}$ in Assumption \ref{asm:map}.
\end{theorem}

We present a proof that the search direction \eqref{eq:p-search} satisfies the sufficient descent condition \eqref{eq:s-descent},
which involves generalizing the Euclidean case in \cite[Proposition 3.1]{nakayama2018hybrid} and \cite[Proposition 2.1]{nakayama2019memoryless}.
\begin{proposition}\label{prop:sufficient}
    Assume that $0<\underline{\gamma}\leq\gamma_{k-1}$ and $0\leq\phi_{k-1}\leq\overline{\phi}^2$ hold, where $1<\overline{\phi}<2$.
    The search direction \eqref{eq:p-search} with
    \begin{align}\label{eq:xi-conditions}
        \begin{cases}
            0\leq\xi_{k-1}\leq\overline{\xi},&\text{if }0\leq\phi_{k-1}\leq 1,\\
            0\leq\xi_{k-1}<\dfrac{\overline{\phi}}{\sqrt{\phi_{k-1}}}-1,&\text{otherwise},
        \end{cases}
    \end{align}
    where $0\leq\overline{\xi}<1$,
    satisfies the sufficient descent condition \eqref{eq:s-descent}
    with
    \begin{align*}
        \kappa:=\min\left\{\frac{3\underline{\gamma}(1-\overline{\xi})}{4},\underline{\gamma}\left(1-\frac{\overline{\phi}^2}{4}\right)\right\}>0.
    \end{align*}
    .
\end{proposition}
\begin{proof}
    The proof involves extending the discussion in \cite[Proposition 3.1]{nakayama2018hybrid} to the case of Riemannian manifolds.
    For convenience, let us set
    \begin{align*}
        \Phi_{k-1}=\frac{1}{\gamma_{k-1}\tau_{k-1}}+\phi_{k-1}\frac{z_{k-1}^\flat z_{k-1}}{s_{k-1}^\flat z_{k-1}}.
    \end{align*}
    From the definition of the search direction \eqref{eq:s-descent}, we have
    \begin{align*}
        \langle g_k,\eta_k\rangle_{x_k}&=-\gamma_{k-1}\lVert g_k\rVert_{x_k}^2+\gamma_{k-1}(1+\xi_{k-1})\phi_{k-1}\frac{(z_{k-1}^\flat g_k)(s_{k-1}^\flat g_k)}{s_{k-1}^\flat z_{k-1}}\\
        &-\gamma_{k-1}\Phi_{k-1}\frac{(s_{k-1}^\flat g_k)^2}{s_{k-1}^\flat z_{k-1}}+\gamma_{k-1}\xi_{k-1}(1-\phi_{k-1})\frac{(z_{k-1}^\flat g_k)^2}{z_{k-1}^\flat z_{k-1}}.
    \end{align*}
    From the relation $2\langle u,v\rangle\leq\lVert u\rVert^2+\lVert v\rVert^2$
    for any vectors $u$ and $v$ in an inner product space, we obtain
    \begin{align*}
        \langle g_k,\eta_k\rangle_{x_k}&\leq-\gamma_{k-1}\lVert g_k\rVert_{x_k}^2+\frac{\gamma_{k-1}\phi_{k-1}}{2}\left(\left\lVert\frac{\sqrt{2}s_{k-1}^\flat g_k}{s_{k-1}^\flat z_{k-1}}z_{k-1}\right\rVert_{x_k}^2+\left\lVert\frac{1+\xi_{k-1}}{\sqrt{2}}g_k\right\rVert_{x_k}^2\right)\\
        &-\gamma_{k-1}\Phi_{k-1}\frac{(s_{k-1}^\flat g_k)^2}{s_{k-1}^\flat z_{k-1}}+\gamma_{k-1}\xi_{k-1}(1-\phi_{k-1})\frac{(z_{k-1}^\flat g_k)^2}{z_{k-1}^\flat z_{k-1}}\\
        &=-\gamma_{k-1}\left(1-\frac{\phi_{k-1}(1+\xi_{k-1})^2}{4}\right)\lVert g_k\rVert_{x_k}^2-\frac{1}{\tau_{k-1}}\frac{(s_{k-1}^\flat g_k)^2}{s_{k-1}^\flat z_{k-1}}\\
        &+\gamma_{k-1}\xi_{k-1}(1-\phi_{k-1})\frac{(z_{k-1}^\flat g_k)^2}{z_{k-1}^\flat z_{k-1}}.
    \end{align*}
    From \eqref{eq:z-1} (i.e., $0<\underline{\nu}\lVert s_{k-1}\rVert_{x_k}^2\leq s_{k-1}^\flat z_{k-1}$), we have
    \begin{align*}
        \langle g_k,\eta_k\rangle_{x_k}\leq-\gamma_{k-1}\left(1-\frac{\phi_{k-1}(1+\xi_{k-1})^2}{4}\right)\lVert g_k\rVert_{x_k}^2 \\
        +\gamma_{k-1}\xi_{k-1}(1-\phi_{k-1})\frac{(z_{k-1}^\flat g_k)^2}{z_{k-1}^\flat z_{k-1}}.
    \end{align*}
    Here, we consider the case $0\leq\phi_{k-1}\leq 1$.
    From $\xi_{k-1}(1-\phi_{k-1})\geq 0$ and the Cauchy-Schwarz inequality,
    we obtain
    \begin{align*}
        \langle g_k,\eta_k\rangle_{x_k}\leq-\gamma_{k-1}\left((1-\xi_{k-1})\left(1-\frac{\phi_{k-1}}{4}(1-\xi_{k-1})\right)\right)\lVert g_k\rVert_{x_k}^2.
    \end{align*}
    From $0\leq\xi_{k-1}\leq\overline{\xi}<1$ and $0\leq\underline{\gamma}\leq\gamma_{k-1}$, we have
    \begin{align*}
        \langle g_k,\eta_k\rangle_{x_k}&\leq-\underline{\gamma}\left((1-\overline{\xi})\left(1-\frac{1}{4}(1-0)\right)\right)\lVert g_k\rVert_{x_k}^2=-\frac{3\underline{\gamma}(1-\overline{\xi})}{4}\lVert g_k\rVert_{x_k}^2.
    \end{align*}
    Next, let us consider the case $1<\phi_{k-1}<\overline{\phi}$.
    From $\xi_{k-1}(1-\phi_{k-1})\leq 0$ and $0\leq\underline{\gamma}\leq\gamma_{k-1}$, we obtain
    \begin{align*}
        \langle g_k,\eta_k\rangle_{x_k}&\leq-\gamma_{k-1}\left(1-\frac{\phi_{k-1}(1+\xi_{k-1})^2}{4}\right)\lVert g_k\rVert_{x_k}^2=-\underline{\gamma}\left(1-\frac{\overline{\phi}^2}{4}\right)\lVert g_k\rVert_{x_k}^2.
    \end{align*}
    Therefore, the search direction \eqref{eq:p-search} satisfies the sufficient descent condition \eqref{eq:s-descent}, i.e.,
    $\langle g_k,\eta_k\rangle_{x_k}\leq-\kappa\lVert g_k\rVert_{x_k}^2$,
    where
    \begin{align*}
        \kappa:=\min\left\{\frac{3\underline{\gamma}(1-\overline{\xi})}{4},\underline{\gamma}\left(1-\frac{\overline{\phi}^2}{4}\right)\right\}>0.
    \end{align*}
    \qed
\end{proof}

Now let us show the global convergence of Algorithm \ref{alg:memoryless}.
\begin{theorem}\label{thm:convergence}
    Suppose that Assumptions \ref{asm:map}, \ref{asm:zoutendijk}
    and \ref{asm:bound} are satisfied.
    Assume further that $0<\underline{\gamma}\leq\gamma_{k-1}\leq \overline{\gamma}$,
    $\underline{\tau}\leq\tau_{k-1}$ and $0\leq\phi_{k-1}\leq\overline{\phi}^2$ hold, where
    $\underline{\tau}>0$ and $1<\overline{\phi}<2$. Moreover, suppose that $\xi_k\in[0,1]$ satisfies \eqref{eq:xi-conditions}.
    Let $(x_k)_{k=0,1,\cdots}$ be a sequence generated by Algorithm \ref{alg:memoryless}, and let the step size $\alpha_k$ satisfy the $\mathscr{T}^{(k)}$-Wolfe conditions \eqref{eq:armijo} and \eqref{eq:wolfe}.
    Then, Algorithm \ref{alg:memoryless} converges in the sense that
    \begin{align*}
        \liminf_{k\to\infty}\lVert g_k\rVert_{x_k}=0
    \end{align*}
    holds.
\end{theorem}
\begin{proof}
    For convenience, let us set
    \begin{align*}
        \Phi_{k-1}=\frac{1}{\gamma_{k-1}\tau_{k-1}}+\phi_{k-1}\frac{z_{k-1}^\flat z_{k-1}}{s_{k-1}^\flat z_{k-1}}.
    \end{align*}
    From \eqref{eq:p-search} and the triangle inequality, we have
    \begin{align*}
        \lVert\eta_k\rVert_{x_k}&=\left\lVert-\gamma_{k-1}g_k+\gamma_{k-1}\left(\phi_{k-1}\frac{z_{k-1}^\flat g_k}{s_{k-1}^\flat z_{k-1}}-\Phi_{k-1}\frac{s_{k-1}^\flat g_k}{s_{k-1}^\flat z_{k-1}}\right)s_{k-1}\right.\\
        &\left.+\gamma_{k-1}\xi_{k-1}\left(\phi_{k-1}\frac{s_{k-1}^\flat g_k}{s_{k-1}^\flat z_{k-1}}+\left(1-\phi_{k-1}\right)\frac{z_{k-1}^\flat g_k}{z_{k-1}^\flat z_{k-1}}\right)z_{k-1}\right\rVert_{x_k}\\ 
        &\leq\gamma_{k-1}\lVert g_k\rVert_{x_k}+\gamma_{k-1}\left\lvert\phi_{k-1}\frac{z_{k-1}^\flat g_k}{s_{k-1}^\flat z_{k-1}}\right\rvert\left\lVert s_{k-1}\right\rVert_{x_k}\\ 
        &+\gamma_{k-1}\left\lvert\Phi_{k-1}\frac{s_{k-1}^\flat g_k}{s_{k-1}^\flat z_{k-1}}\right\lvert\left\lVert s_{k-1}\right\rVert_{x_k}
        +\gamma_{k-1}\xi_{k-1}\left\lvert\phi_{k-1}\frac{s_{k-1}^\flat g_k}{s_{k-1}^\flat z_{k-1}}\right\rvert \lVert z_{k-1}\rVert_{x_k}\\
        &+\gamma_{k-1}\xi_{k-1}\left\lvert\left(1-\phi_{k-1}\right)\frac{z_{k-1}^\flat g_k}{z_{k-1}^\flat z_{k-1}}\right\rvert \left\lVert z_{k-1}\right\rVert_{x_k}.
    \end{align*}
    Here, from the Cauchy-Schwarz inequality, we obtain
    \begin{align*}
        \lvert z_{k-1}^\flat g_k\rvert
        =\lvert\langle z_{k-1}, g_k\rangle_{x_k}\rvert
        \leq\lVert z_{k-1}\rVert_{x_k}\lVert g_k\rVert_{x_k}, \\
        \lvert s_{k-1}^\flat g_k\rvert
        =\lvert\langle s_{k-1}, g_k\rangle_{x_k}\rvert
        \leq\lVert s_{k-1}\rVert_{x_k}\lVert g_k\rVert_{x_k},
    \end{align*}
    
    which together with \eqref{eq:z-1} (i.e., $\underline{\nu}\lVert s_{k-1}\rVert_{x_k}^2\leq s_{k-1}^\flat z_{k-1}$) and
    $0\leq\phi_{k-1}<4$, gives
    \begin{align*}
        \lVert\eta_k\rVert_{x_k}&\leq\gamma_{k-1}\lVert g_k\rVert_{x_k}+4\gamma_{k-1}\frac{\lVert z_{k-1}\rVert_{x_k}\lVert g_k\rVert_{x_k}}{\underline{\nu}\lVert s_{k-1}\rVert_{x_k}^2}\lVert s_{k-1}\rVert_{x_k}\\
        &+\gamma_{k-1}\left(\frac{1}{\gamma_{k-1}\tau_{k-1}}+\frac{4\lVert z_{k-1}\rVert_{x_k}^2}{\underline{\nu}\lVert s_{k-1}\rVert_{x_k}^2}\right)\frac{\lVert s_{k-1}\rVert_{x_k}\lVert g_k\rVert_{x_k}}{\underline{\nu}\lVert s_{k-1}\rVert_{x_k}^2}\lVert s_{k-1}\rVert_{x_k}\\
        &+\gamma_{k-1}\xi_{k-1}\frac{4\lVert s_{k-1}\rVert_{x_k}\lVert g_k\rVert_{x_k}}{\underline{\nu}\lVert s_{k-1}\rVert_{x_k}^2}\lVert z_{k-1}\rVert_{x_k}+\gamma_{k-1}\xi_{k-1}\frac{3\lVert z_{k-1}\rVert_{x_k}\lVert g_k\rVert_{x_k}}{\lVert z_{k-1}\rVert_{x_k}^2}\lVert z_{k-1}\rVert_{x_k}.
    \end{align*}
    From \eqref{eq:z-2} (i.e., $\lVert z_{k-1}\rVert_{x_k}\leq \overline{\nu}\lVert s_{k-1}\rVert_{x_k}$),
    $0\leq\xi_{k-1}\leq 1$, and $\gamma_{k-1}\leq\overline{\gamma}$, we have
    \begin{align*}
        \lVert\eta_k\rVert_{x_k}&\leq\overline{\gamma}\lVert g_k\rVert_{x_k}+\frac{4\overline{\gamma}\overline{\nu}}{\underline{\nu}}\lVert g_k\rVert_{x_k}+\frac{1}{\underline{\tau}\underline{\nu}}\lVert g_k\rVert_{x_k}+\frac{4\overline{\gamma}\overline{\nu}^2}{\underline{\nu}^2}\lVert g_k\rVert_{x_k}\\
        &+\frac{4\overline{\gamma}\overline{\nu}}{\underline{\nu}}\lVert g_k\rVert_{x_k}+3\overline{\gamma}\lVert g_k\rVert_{x_k},
    \end{align*}
    which, together with $\lVert g_k\rVert_{x_k}\leq\Gamma$, give
    \begin{align*}
        \lVert\eta_k\rVert_{x_k}\leq\underbrace{\left(4\overline{\gamma}+
        \frac{8\overline{\gamma}\overline{\nu}}{\underline{\nu}}+\frac{1}{\underline{\tau}\underline{\nu}}+\frac{4\overline{\gamma}\overline{\nu}^2}{\underline{\nu}^2}\right)}_{\Theta}\Gamma=\Theta\Gamma.
    \end{align*}
    To prove convergence by contradiction, suppose that there exists a positive constant $\varepsilon>0$ such that
    \begin{align*}
        \lVert g_k\rVert_{x_k}\geq\varepsilon,
    \end{align*}
    for all $k$. From Proposition \ref{prop:sufficient},
    \begin{align*}
        \langle g_k,\eta_k\rangle_{x_k}\leq-\kappa\lVert g_k\rVert_{x_k}^2\leq-\kappa\varepsilon^2,
    \end{align*}
    where
    \begin{align*}
        \kappa:=\min\left\{\frac{3\underline{\gamma}(1-\overline{\xi})}{4},\underline{\gamma}\left(1-\frac{\overline{\phi}^2}{4}\right)\right\}>0.
    \end{align*}
    It follows from the above inequalities that 
    \begin{align*}
        \infty=
        \sum_{k=0}^\infty\frac{\kappa^2\varepsilon^4}{\Theta^2\Gamma^2}\leq
        \sum_{k=0}^\infty\frac{\langle g_k,\eta_k\rangle_{x_k}^2}{\lVert \eta_k\rVert_{x_k}^2}.
    \end{align*}
This contradicts the Zoutendijk theorem (Theorem \ref{thm:zoutendijk}) and thus completes the proof.\qed
\end{proof}

\section{Numerical experiments}\label{sec:numerical}
We compared the proposed method with existing methods, including the Riemannian conjugate gradient methods and memoryless spectral-scaling Broyden family.
In the experiments, we implemented the proposed method as an optimizer of \texttt{pymanopt} (see \cite{townsend2016pymanopt}) and solved two Riemannian optimization problems (Problems \ref{pbl:rayleigh} and \ref{pbl:off-diag}).

Problem \ref{pbl:rayleigh} is the Rayleigh-quotient minimization problem on the unit sphere \cite[Chapter 4.6]{absil2008optimization}.
\begin{problem}\label{pbl:rayleigh}
Let $A\in\mathbb{R}^{n\times n}$ be a symmetric matrix,
    \begin{align*}
        \text{minimize }&f(x):=x^\top Ax, \\
        \text{subject to }&x\in\mathbb{S}^{n-1}:=\{x\in\mathbb{R}^n:\lVert x\rVert=1\},
    \end{align*}
where $\lVert\cdot\rVert$ denotes the Euclidean norm.
\end{problem}
In the experiments, we set $n=100$ and generated a matrix $B\in\mathbb{R}^{n\times n}$ with randomly chosen elements by using \texttt{numpy.random.randn}.
Then, we set a symmetric matrix $A=(B+B^\top)/2$.

Absil and Gallivan \cite[Section 3]{absil2006joint} introduced an off-diagonal cost function.
Problem \ref{pbl:off-diag} is an off-diagonal cost function minimization problem on an oblique manifold.
\begin{problem}\label{pbl:off-diag}
Let $C_i\in\mathbb{R}^{n\times n}$ $(i=1,\cdots,N)$ be symmetric matrices.
    \begin{align*}
        \text{minimize }&f(X):=\sum_{i=1}^N\lVert X^\top C_iX-\mathrm{ddiag}(X^\top C_iX)\rVert_F^2, \\
        \text{subject to }&X\in\mathcal{OB}(n,p):=\{X\in\mathbb{R}^{n\times p}:\mathrm{ddiag}(X^\top X)=I\},
    \end{align*}
where $\lVert\cdot\rVert_F$ denotes the Frobenius norm and $\mathrm{ddiag}(M)$ denotes a diagonal matrix $M$ with all its off-diagonal elements set to zero.
\end{problem}
In the experiments, we set $N=5$, $n=10$, and $p=5$ and generated five matrices $B_i\in\mathbb{R}^{n\times n}$ $(i=1,\cdots,5)$
with randomly chosen elements by using \texttt{numpy.random.randn}.
Then, we set symmetric matrices $C_i=(B_i+B_i^\top)/2$ $(i=1,\cdots,5)$.

The experiments used a MacBook Air (M1, 2020) with version 12.2 of the macOS Monterey operating system.
The algorithms were written in Python 3.11.3 with the NumPy 1.25.0 package and the Matplotlib 3.7.1 package.
Python implementations of the methods used in the numerical experiments are available at \texttt{https://github.com/iiduka-researches/202307-memoryless}.

We considered that a sequence had converged to an optimal solution when the stopping condition,
\begin{align*}
    \lVert\mathrm{grad}f(x_k)\rVert_{x_k}<10^{-6},
\end{align*}
was satisfied.
We set $\mathscr{T}^{(k-1)}=\mathcal{T}_{\alpha_{k-1}\eta_{k-1}}^R(\cdot)$,
\begin{align*}
    \gamma_{k-1}=\max\left\{1,\frac{s_{k-1}^\flat z_{k-1}}{z_{k-1}^\flat z_{k-1}}\right\}, \quad
    \tau_{k-1}=\min\left\{1,\frac{z_{k-1}^\flat z_{k-1}}{s_{k-1}^\flat z_{k-1}}\right\}.
\end{align*}
We compared the proposed methods with the existing Riemannian optimization algorithms,
including Riemannian conjugate gradient methods. Moreover, we compared twelve versions of the proposed method with different parameters,
i.e., $\phi_{k-1}$, $z_{k-1}$ and $\xi_{k-1}$. We compared the BFGS formula $\phi_{k-1}=1$ and the preconvex class $\phi_{k-1}\in[0,\infty)$.
For the preconvex class (see \cite[(43)]{narushima2023memoryless}), we used
\begin{align*}
    \phi_{k-1}=\frac{0.1\theta^\ast_{k-1}-1}{0.1\theta^\ast_{k-1}(1-\mu_{k-1})-1},
\end{align*}
where
\begin{align*}
    \theta_{k-1}^\ast=\max\left\{\frac{1}{1-\mu_{k-1}},10^{-5}\right\},
    \quad
    \mu_{k-1}=\frac{(s_{k-1}^\flat s_{k-1})(z_{k-1}^\flat z_{k-1})}{(s_{k-1}^\flat z_{k-1})^2}.
\end{align*}
Moreover, we compared Li-Fukushima regularization \eqref{eq:li-fukushima-1} and \eqref{eq:li-fukushima-2} with $\hat{\nu}=10^{-6}$
and Powell’s damping technique \eqref{eq:powell-1} and \eqref{eq:powell-2} with $\hat{\nu}=0.1$.
In addition, we used a constant parameter $\xi_{k-1}=\xi\in[0,1]$ and compared our methods with $\xi=1$ (i.e., the existing methods when $\mathscr{T}^{(k)}=\mathcal{T}_{\alpha_{k-1}\eta_{k-1}}(\cdot)$),
$\xi=0.8$, and $\xi=0.1$. For comparison, we also tested two Riemannian conjugate gradient methods,
i.e., DY \eqref{eq:DY} and HZ \eqref{eq:HZ}.

As the measure for these comparisons, we calculated the performance profile $P_s:\mathbb{R}\to[0,1]$ \cite{dolan2002benchmarking} defined as follows:
let $\mathcal{P}$ and $\mathcal{S}$ be the sets of problems and solvers, respectively.
For each $p\in\mathcal{P}$ and $s\in\mathcal{S}$, 
\begin{align*}
    t_{p,s}:=(\text{iterations or time required to solve problem }p\text{ by solver }s).
\end{align*}
We defined the performance ratio $r_{p,s}$ as
\begin{align*}
    r_{p,s}:=\frac{t_{p,s}}{\min_{s^\prime\in\mathcal{S}}t_{p,s^\prime}}.
\end{align*}
Next, we defined the performance profile $P_s$ for all $\tau\in\mathbb{R}$ as
\begin{align*}
    P_s(\tau):=\frac{\lvert\{p\in\mathcal{P}:r_{p,s}\leq\tau\}\rvert}{\lvert\mathcal{P}\rvert},
\end{align*}
where $\lvert A\rvert$ denotes the number of elements in a set $A$.
In the experiments, we set $\lvert\mathcal{P}\rvert=100$ for Problems \ref{pbl:rayleigh} and \ref{pbl:off-diag}, respectively.

Figures \ref{fig:rayleigh_f}--\ref{fig:off_p} plot the results of our experiments.
In particular, Figure \ref{fig:rayleigh_f} shows the numerical results for Problem \ref{pbl:rayleigh} with Li-Fukushima regularization \eqref{eq:li-fukushima-1} and \eqref{eq:li-fukushima-2}.
It shows that Algorithm \ref{alg:memoryless} with $\xi=0.1$ has much higher performance than that of Algorithm \ref{alg:memoryless} with $\xi=1$
(i.e., the existing method) regardless of whether the BFGS formula or the preconvex class is used.
In addition, we can see that Algorithm \ref{alg:memoryless} with $\xi=0.8$ and $\xi=1$ have about the same performance.

\begin{figure}[htbp]
\centering
\subfigure[iteration]{\includegraphics[width=55mm]{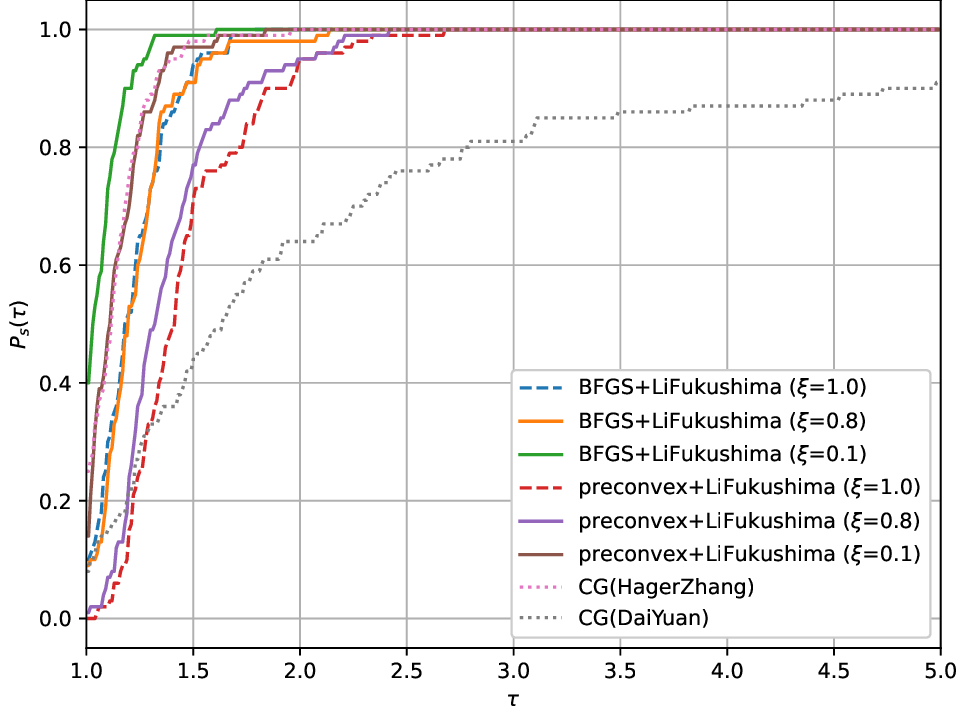}}
\subfigure[elapsed time]{\includegraphics[width=55mm]{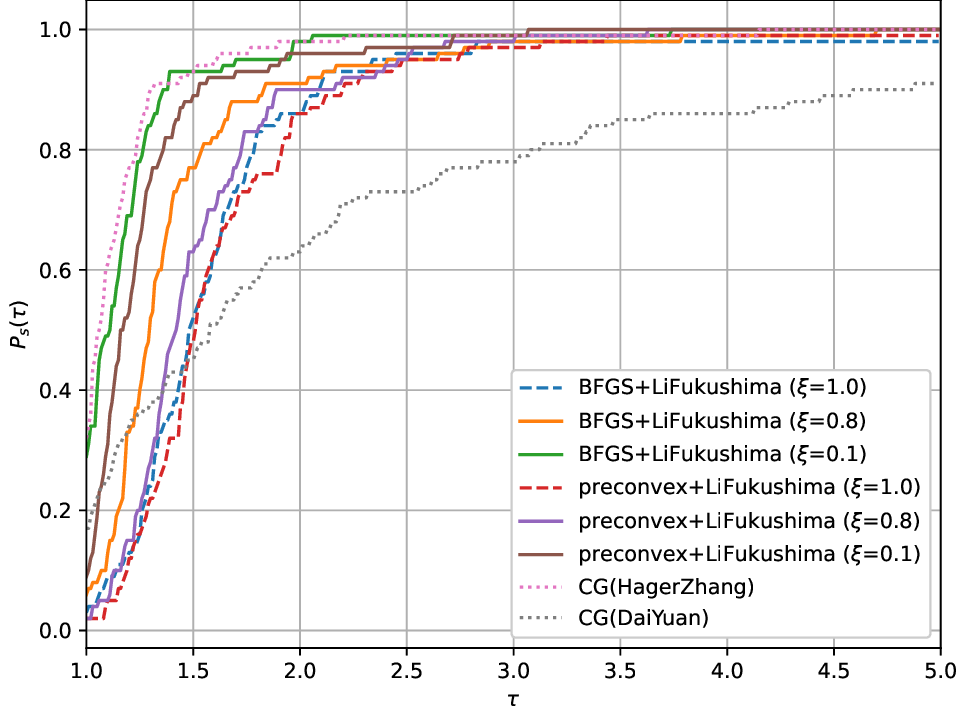}}
\caption{Performance profiles of each algorithm versus the number of iterations (a) and the elapsed time (b) for Problem \ref{pbl:rayleigh}. $z_k$ is defined by Li-Fukushima regularization \eqref{eq:li-fukushima-1} and \eqref{eq:li-fukushima-2}.\label{fig:rayleigh_f}}
\end{figure}

Figure \ref{fig:rayleigh_p} shows the numerical results
for solving Problem \ref{pbl:rayleigh} with Powell’s damping technique \eqref{eq:powell-1} and \eqref{eq:powell-2}.
It shows that Algorithm \ref{alg:memoryless} with $\xi=0.1$ is superior to Algorithm \ref{alg:memoryless} with $\xi=1$ (i.e., the existing method),
regardless of whether the BFGS formula or the preconvex class is used.
Moreover, it can be seen that Algorithm \ref{alg:memoryless} with $\xi=0.8$ and $\xi=1$ has about the same performance.

\begin{figure}[htbp]
\centering
\subfigure[iteration]{\includegraphics[width=55mm]{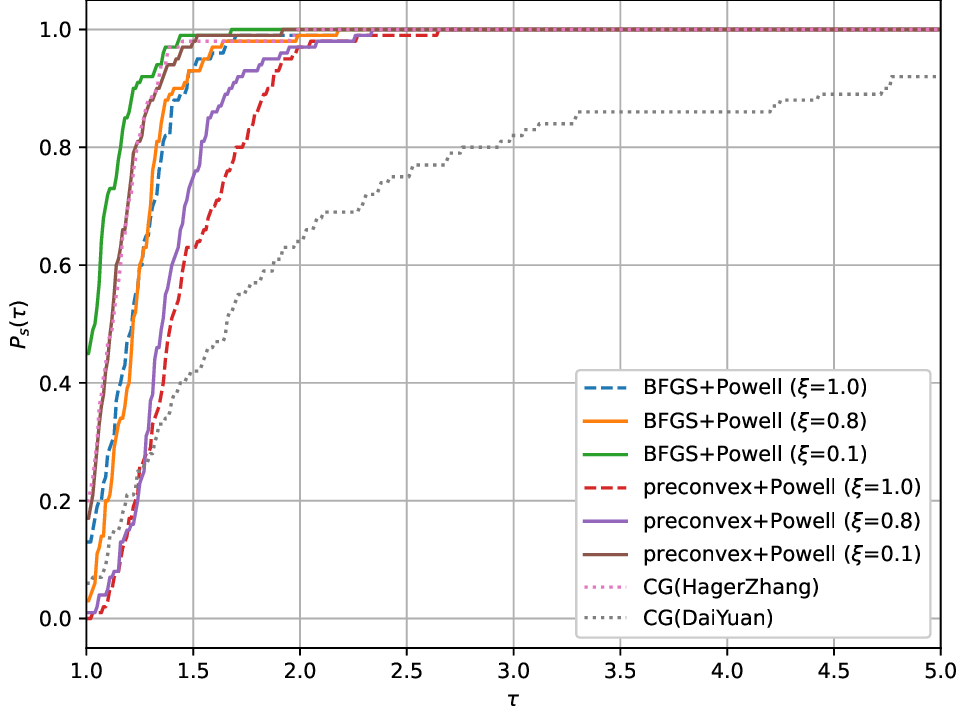}}
\subfigure[elapsed time]{\includegraphics[width=55mm]{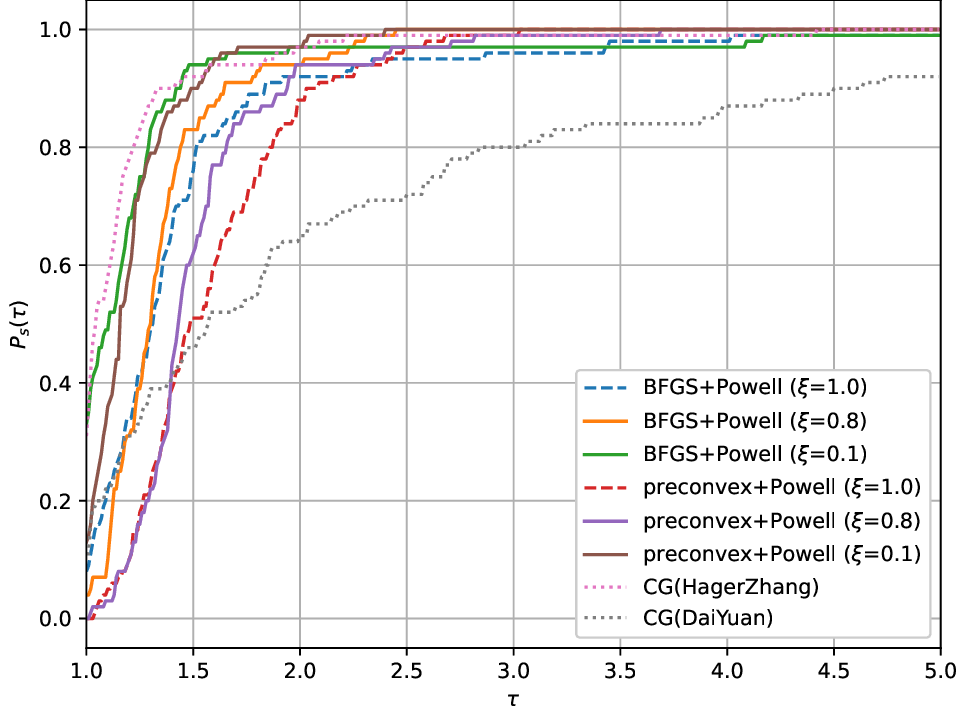}}
\caption{Performance profiles of each algorithm versus the number of iterations (a) and the elapsed time (b) for Problem \ref{pbl:rayleigh}. $z_k$ is defined by Powell’s damping technique \eqref{eq:powell-1} and \eqref{eq:powell-2}. \label{fig:rayleigh_p}}
\end{figure}

Figure \ref{fig:off_f} shows numerical results for Problem \ref{pbl:rayleigh} with Li-Fukushima regularization \eqref{eq:li-fukushima-1} and \eqref{eq:li-fukushima-2}.
It shows that if we use the BFGS formula (i.e., $\phi_k=1$), then Algorithm \ref{alg:memoryless} with $\xi=0.8$ and the HZ method outperform the others.
However, Algorithm \ref{alg:memoryless} with the preconvex class is not compatible with is an off-diagonal cost function minimization problem on an oblique manifold.

\begin{figure}[htbp]
\centering
\subfigure[iteration]{\includegraphics[width=55mm]{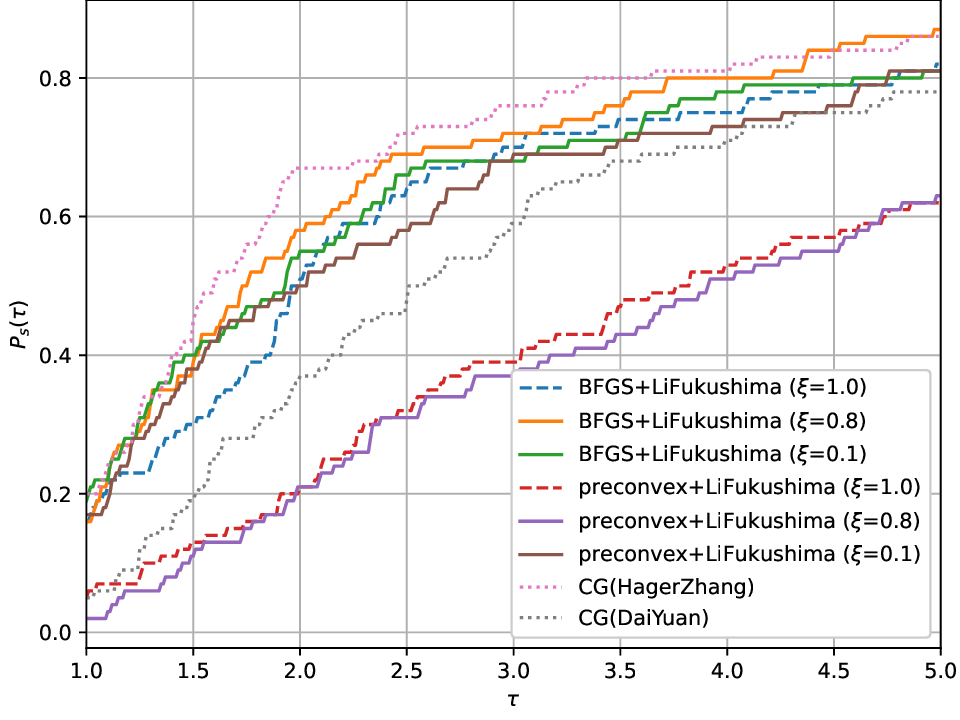}}
\subfigure[elapsed time]{\includegraphics[width=55mm]{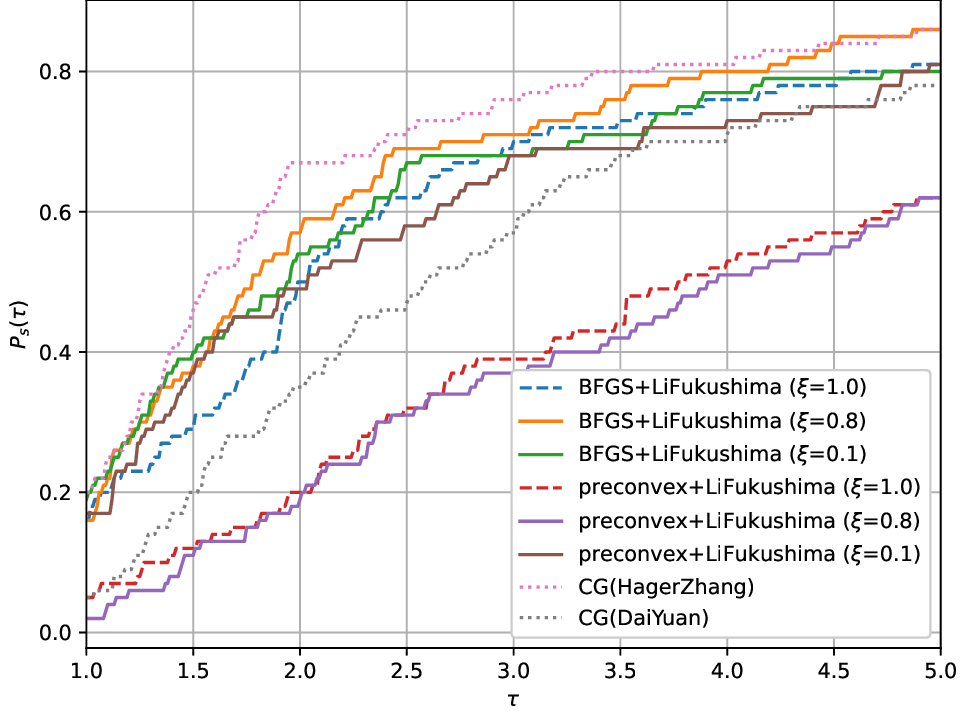}}
\caption{Performance profiles of each algorithm versus the number of iterations (a) and the elapsed time (b) for Problem \ref{pbl:off-diag}. $z_k$ is defined by Li-Fukushima regularization \eqref{eq:li-fukushima-1} and \eqref{eq:li-fukushima-2}. \label{fig:off_f}}
\end{figure}

Figure \ref{fig:off_p} shows the numerical results for solving Problem \ref{pbl:rayleigh} with Powell’s damping technique \eqref{eq:powell-1} and \eqref{eq:powell-2}.
It shows that if we use the BFGS formula (i.e., $\phi_k=1$), then Algorithm \ref{alg:memoryless} with $\xi=0.8$ or $\xi=1$ is superior to the others.
However, Algorithm \ref{alg:memoryless} with the preconvex class is not compatible with is an off-diagonal cost function minimization problem on an oblique manifold.
Therefore, we can see that Algorithm \ref{alg:memoryless} with the BFGS formula (i.e., $\phi_k=1$) is suitable for solving an off-diagonal cost function minimization problem on an oblique manifold.

\begin{figure}[htbp]
\centering
\subfigure[iteration]{\includegraphics[width=55mm]{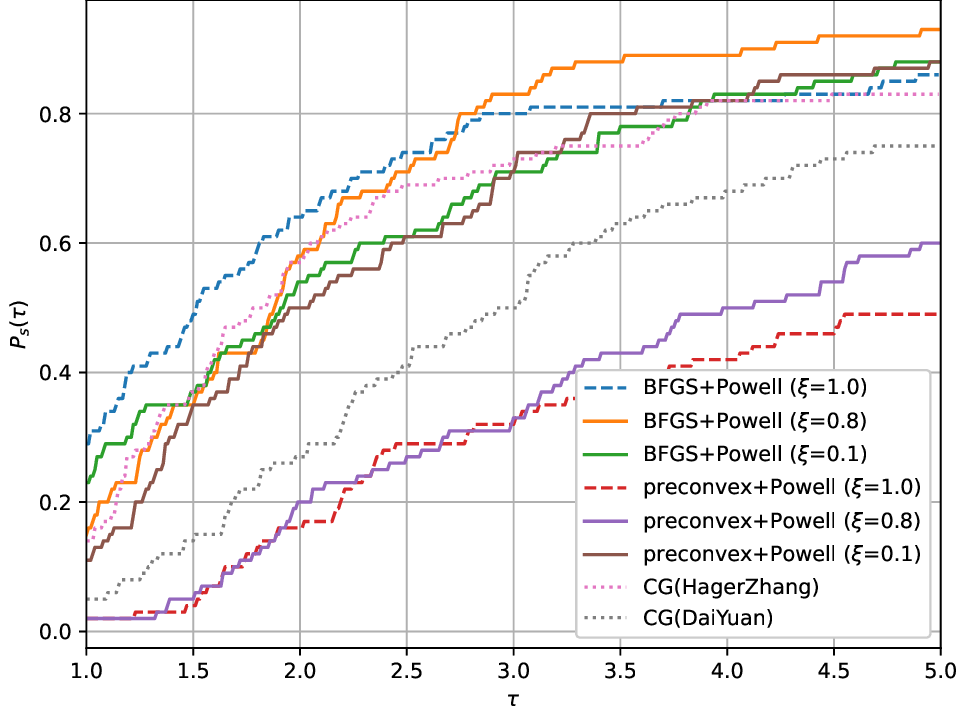}}
\subfigure[elapsed time]{\includegraphics[width=55mm]{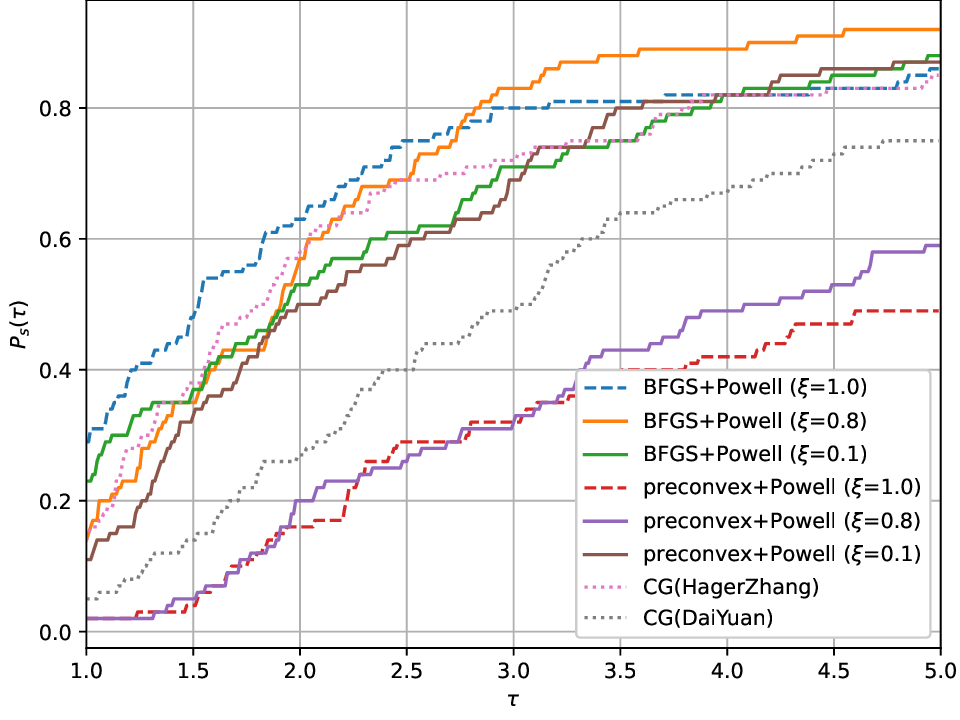}}
\caption{Performance profiles of each algorithm versus the number of iterations (a) and the elapsed time (b) for Problem \ref{pbl:off-diag}. $z_k$ is defined by Powell’s damping technique \eqref{eq:powell-1} and \eqref{eq:powell-2} \label{fig:off_p}}
\end{figure}

\section{Conclusion}\label{sec:conclusion}
This paper presented a modified memoryless quasi-Newton method with the spectral-scaling Broyden family on Riemannian manifolds,
i.e., Algorithm \ref{alg:memoryless}. Algorithm \ref{alg:memoryless} is a generalization of the memoryless self-scaling Broyden family on Riemannian manifolds.
Specifically, it involves adding one parameter to the search direction.
We use a general map instead of vector transport, similarly to the general framework of Riemannian conjugate gradient methods.
Therefore, we can utilize methods that use vector transport, scaled vector transport, or an inverse retraction.
Moreover, we proved that the search direction satisfies the sufficient descent condition, and the method globally converges under the Wolfe conditions.
Moreover, the numerical experiments indicated that the proposed method with the BFGS formula is suitable for solving an off-diagonal cost function minimization problem on an oblique manifold.

\section*{Acknowledgements}
This work was supported by a JSPS KAKENHI Grant Number JP23KJ2003.

\section*{Data availability statements}
Data sharing is not applicable to this article as no datasets were generated or analysed during the current study.

\bibliographystyle{abbrv}
\bibliography{biblib.bib}

\end{document}